\documentclass[11pt]{article}
\usepackage[utf8]{inputenc}
\usepackage[letterpaper]{geometry}
\usepackage{amsfonts, mathtools, amsmath, amssymb, amsthm}
\usepackage{hyperref, bookmark, tabulary, xstring, cleveref, keytheorems, multicoltab}
\usepackage[all]{hypcap}
\usepackage[backend=biber, style=alphabetic]{biblatex}
\usepackage{listings, xcolor, authblk, enumitem, multicol}
\usepackage{graphicx}
\graphicspath{{images/}}

\title{On the evenness of coefficients of modular forms}
\author{Sayantan Santra}
\affil{Department of Mathematics, University of Oklahoma, \\ Norman, OK 73019, United States}
\date{}
\newkeytheorem{theorem}

\newkeytheorem{corollary}

\newkeytheorem{lemma}

\newkeytheorem{definition}

\newcommand{\suminf}{\sum_{n \ge 1}}
\newcommand{\ZZ}{\mathbb{Z}}
\newcommand{\QQ}{\mathbb{Q}}

\newcommand{\FF}{\mathbb{F}}

\newcommand{\modrep}[1]{\tilde{\rho}_{#1}}
\newcommand{\mf}{\mathfrak}
\newcommand{\mc}{\mathcal}

\newcommand{\gl}[1]{GL_2(#1_\ell)}
\newcommand{\absgalgrp}{Gal(\bar{\QQ}/\QQ)}

\newcommand{\say}[1]{``#1"}

\begin{document}
\maketitle

\begin{abstract}
	We study the density of primes such that \(a_p\) is divisible by a prime above \(2\) for modular forms \(f = \suminf a_n q^n\) of small
	levels. First, we reduce the problem to low weight and odd level, or a mod \(2\) CM type form. Then, we provide calculations for the
	evenness of particular newforms.
	In particular, we provide a list of levels where this \say{evenness} can be computed exactly for all the forms with trivial
	nebentypus.
\end{abstract}

\section{Introduction}
Let \(f(z) = \suminf a_n q^n\) be a newform of level \(N\), weight \(k\), and trivial nebentypus.
We fix a prime \(\lambda\) lying over \(2\) in the form's ring of coefficients, and define \(\lambda\)-evenness as divisibility by that
prime. (We usually omit the \(\lambda\) when it can be understood from context.)

From numerical data, it can be observed that for \(N = 2^s M\) where \(M\) is a prime, there are a lot of even coefficients \(a_p\)
where \(p\) is a prime away from the level. In particular for small primes \(M = 2, 3, 5\), and \(7\), all of the \(a_p\)s are even.

The latter observation is explained using the Deligne--Serre construction (\autoref{modrep_def}) of a mod \(2\) representation attached
to a naive mod \(2\) modular form. Here, by \say{naive mod \(2\) modular forms}, we mean the reduction of a elliptic modular form modulo
some prime \(\lambda\) lying over \(2\) in its coefficient ring.

Using computational data (provided in \autoref{computational-data}), we observe a notable bias towards evenness for newforms
of small levels. This can be seen in \autoref{stat_plot} provided on the next page.

While it is not enough to say that this tendency persists for higher levels, it raises an interesting problem for
future research: to what extent does this bias persist for general levels? The methods used in this paper are not sufficient to answer that
question, but further investigation using a larger sample of levels may allow of better insight. It would also be interesting to see what
happens beyond the case of elliptic modular forms e.g. Hilbert modular forms or Siegel modular forms.

Another question worth investigation might be whether there are infinitely many odd levels where all newforms of that level have evenness \(1\).
The observations made in this paper are insufficient to make a strong conjecture in either direction, but \autoref{stat_plot} does seem to
suggest that the answer is affirmative.

\medbreak

\begin{figure}[htbp]
	\noindent\rule{\linewidth}{0.4pt}
	\centering
	\includegraphics[width=0.9\textwidth]{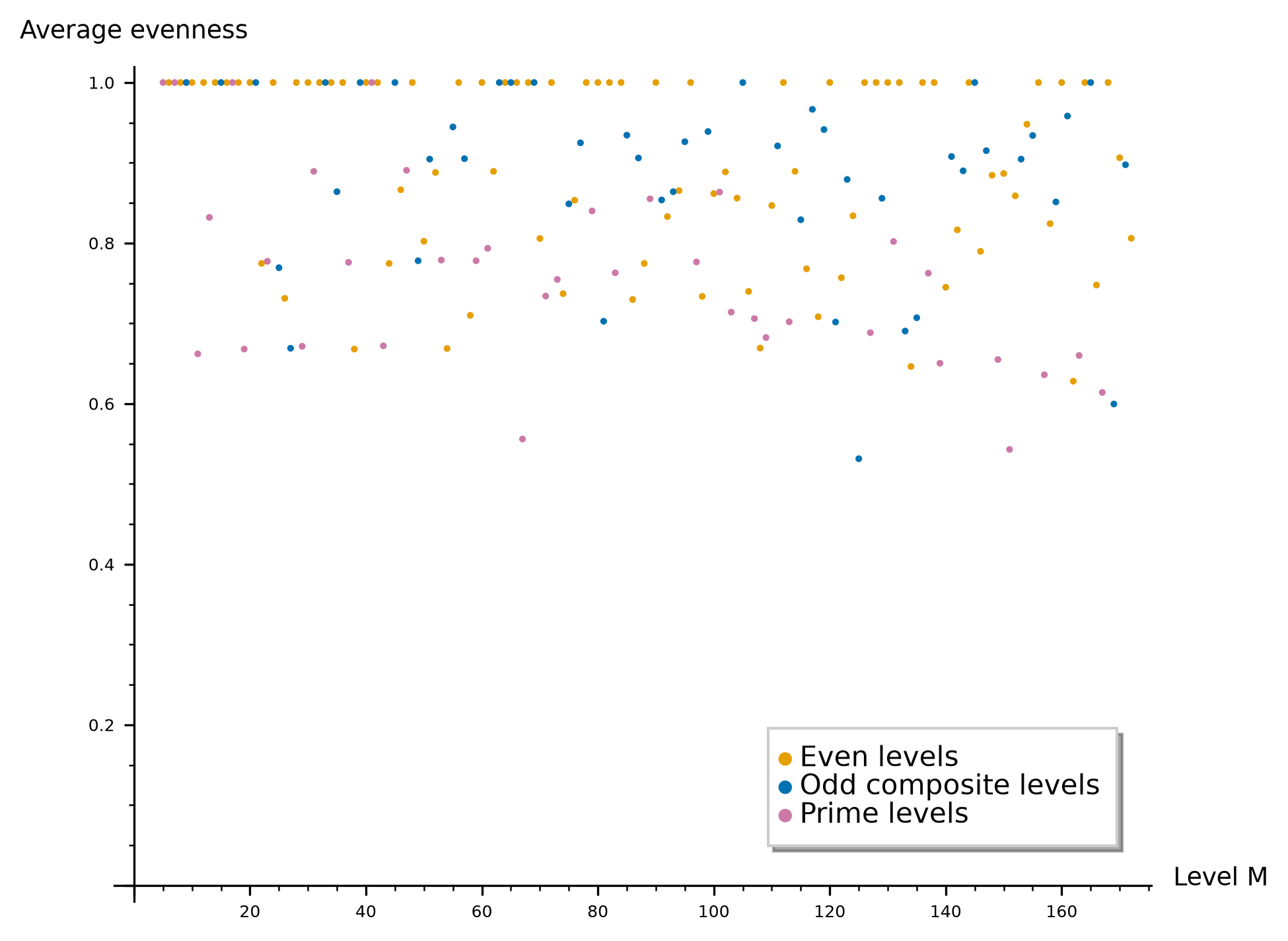}
	\caption{Average evenness of newforms by level. Prime levels are shown in purple, odd composite levels in blue, and even levels in
		orange. This plot illustrates the bias towards evenness for small levels. }
	\label{stat_plot}
	\noindent\rule{\linewidth}{0.4pt}
\end{figure}

\medbreak

The main results of this paper are the following.

\begin{definition}
	For a newform \(f = \suminf a_n q^n\), and a prime \(\lambda \mid 2\), we define its \(\lambda\)-evenness \(\mc E(f)\) as the density of primes
	\(p\) where \(\lambda \mid a_p\).
\end{definition}

\begin{theorem}
	\label{pushdown_theorem}
	Let \(f = \suminf a_n q^n\) be a newform of trivial nebentypus, and level \(2^s M\) for some prime \(M\), and integer \(s \ge 0\).
	Then, one of the following must be true.
	\begin{enumerate}[label=(\alph*)]
		\item \(f\) is congruent modulo 2 to some cusp form \(g\) of level \(M\) or \(2M\), weight 2, and trivial nebentypus; or,
		\item \(f\) is of CM type mod 2 by \(\QQ(i)\) and has \(\mc E(f) \ge 0.5\).
	\end{enumerate}
\end{theorem}
By CM mod 2 by \(\QQ(i)\), we mean forms such that its Fourier coefficients \(a_p\) are congruent to \(0\) modulo \((i+1)\) for every prime
\(p \equiv 3 \pmod{4}\).

It should be noted that this theorem essentially follows from \autocite{hatada_1998}, but no direct references for this were found in the
literature. In that paper, the nebentypus of the form at the lower level wasn't explicitly shown to be trivial, which this paper does.

\begin{corollary}
	\label{small_primes_corollary}
	Let \(M = 2, 3, 5\), or 7 in \autoref{pushdown_theorem}. Then, \(f\) has even coefficients \(a_p\) for \(p \nmid 2 M\).
\end{corollary}

Note that the theorem below covers all odd \(M \le 47\).

\begin{theorem}
	\label{stat_theorem}
	Let \(f\) be a newform of trivial nebentypus, and level \(2^s M\) for some odd number \(M\), and integer \(s \ge 0\).
	\begin{enumerate}[label=(\alph*)]
		\item If \(M \in \{ 1, 3, 5, 7, 9, 15, 17, 21, 33, 39, 41, 45, 65, 69, 105, 145, 165 \}\), then \(\mc E(f) = 1\).
		\item If \(M \in \{ 11, 19, 27, 29, 43 \}\), then \(\mc E(F) = \frac{2}{3}\).
		\item If \(M \in \{ 13, 23, 25, 31, 35, 37, 47, 51, 55, 57, 77, 85, 95, 111, 117, 141, 155 \}\), then
		      \(\mc E(f) = 1\), or \(\mc E(f) = \frac{2}{3}\).
	\end{enumerate}
\end{theorem}

The idea of connecting the Deligne--Serre representation attached to a modular form to its congruence properties was first developed by
Serre and Swinnerton-Dyer. In \autocite{swinnerton_dyer_1973}, Swinnerton-Dyer proved quite a few well-known congruences using Deligne--Serre
representations and categorized all the congruences for level 1 into three types. In \autocite{swinnerton_dyer_1977}, he further improved
these results to congruences modulo powers of odd primes. For a given modular form, the primes modulo which we can have a congruence
relation, are called \say{exceptional} for that form.

In \autocite{ono_1994}, Ono categorized all possible congruences for all levels modulo odd exceptional primes. One thing
missing from all of these papers are results modulo \(2\), due to complications arising in this case. This paper focuses on this special case.

As presented above, \autoref{pushdown_theorem} tells us that, for newforms of levels \(2^s M\), either it arises from a newform of level \(M\) weight 2,
or it is of CM type mod 2 by \(\QQ(i)\). Further, \autoref{small_primes_corollary} tells us that for \(M = 2,3,5,7\), the evenness is always 1.

When we look at newforms at higher levels, we notice unusually high evenness of their coefficients. For example, for levels 1 to
100, evenness \(\ge 0.8\) occurs 69.6\% of the time, and evenness \(\ge 0.5\) occurs 97.6\% of the time. We were able to theoretically
verify the validity of most of the computational data using different means, which we will outline later in the paper.

The work is heavily inspired by works of Hatada \autocite{hatada_1998}, and Ono--Taguchi \autocite{ono_taguchi_2005} about the modular
representations attached to modular forms. They were mostly interested in categorizing these attached representations, whereas we use them
to calculate the evenness.

\section{Preliminaries}
Let \(f\) be a newform of weight \(k\), level \(N\), and trivial nebentypus.

The Deligne--Serre construction \autocite{deligne_1971} allows us to attach a unique, semisimple, continuous, odd, linear representation

\[\rho_{f,\ell} : \absgalgrp \to \gl{\bar{\ZZ}}\] to \(f\).
Further, it is unramified outside \(N \ell\) such that for all primes \(p\) with \(p \nmid \ell N\), we have
\[tr{\rho_{f,\ell}(Frob(p))} = a_p,\]
\[\det{\rho_{f,\ell}(Frob(p))} = p^{k-1}.\]
In other words, the characteristic polynomial of \(\rho_{f,\ell}(Frob(p))\) is \(x^2 - a_p x + p^{k-1}\). Also, \(\chi_\ell (Frob(p)) = p\).

This gives us a way to study the modular forms by studying these representations.

\label{modrep_def}
Let \(\rho_{f,\ell}\) be the representation associated to a modular form \(f\) as described above. We write \(\modrep{f, \lambda}\) to
denote this representation reduced modulo some prime \(\lambda\) lying in the coefficient ring \(\mf O\) of \(f\) i.e.
\[\modrep{f, \lambda} : \absgalgrp \to GL_2(F_\lambda),\]
where \(F_\lambda\) is the residue field of \(\mf O\) at \(\lambda\).
As per \autocite[640]{ribet_1994}, this representation is unique up to semi-simplification.

In many places, we will write this as \(\modrep{f, \ell}\) and \(\lambda\) will be a prime lying over \(\ell\). We can also think of this
representation as being attached to a naive mod \(2\) modular form, i.e. the reduction of a elliptic modular form modulo \(2\).

\section{Proof of \autoref{pushdown_theorem}}
We want to prove the following lemma that is crucial to proving \autoref{pushdown_theorem}.
\begin{lemma}
	\label{level_lemma}
	Let \(f\) be a newform of trivial nebentypus, and level \(\ell^s M\) where \(\ell\) is a prime,
	\(M \ge 1, s \ge 0\), and \((M, \ell) = 1\). Then, there exists a newform \(f'\) of trivial nebentypus, and level \(M\) such that
	\(\modrep {f, \ell} = \modrep {f', \ell}\).
\end{lemma}
\begin{proof}
	As mentioned before, this is a modification of \autocite[Theorem 2]{hatada_1998}.

	Indeed, if we start with a modular form \(f\) with trivial nebentypus \(\phi\), we can write \(\phi = \psi \cdot \xi \cdot \omega\)
	through
	\[(\ZZ / q \ell^s M \ZZ)^* \cong (\ZZ / M \ZZ)^* \times (1 + q\ZZ_{\ell}) / (1 + q \ell^s \ZZ_{\ell}) \times (\ZZ / q\ZZ)^*.\]

	Here, \(q = 4\) if \(\ell = 2\), and \(q = \ell\) otherwise.

	If we follow the computation in the cited paper, we get a newform on \(\ker{\psi}\) of weight \(k + \alpha\). Here, \(\ker{\psi}\) is the
	kernel of \(\psi\) considering it as a character of \(\Gamma_0(M)\), and \(\alpha\) is such that \(\omega(y \mod q) = y^\alpha, \forall y
	\in \mu_c, c = {\left| (\ZZ / q\ZZ)^* \right|}\). So, when \(\phi\) is trivial, we get a newform of trivial nebentypus.
\end{proof}

Hence, to study congruences of newforms in level \(2^s M\), it is enough to study congruences of modular forms in level \(M\). In
particular, for \(\ell = 2\), we get information about the even coefficients of a newform. We write the specialized version of the lemma for
\(\ell = 2\) in the following corollary for ease of reference.

\begin{corollary}
	\label{level_reduction_corollary}
	Let \(f\) be a newform of trivial nebentypus, and level \(2^s M\) where \(M\) is odd, and \(s \ge 0\).
	Then, there exists a newform \(f'\) of trivial nebentypus, and level \(M\) such that \(\modrep {f, 2} = \modrep {f', 2}\).
\end{corollary}

By \autoref{level_reduction_corollary}, we have put some restrictions on the levels of modular forms that we need to study. Now, we would like
to put some restrictions on the weight.

Let \(f\) be a newform as in \autoref{pushdown_theorem}. If \(\modrep f\) be the associated representation for \(\ell=2\), then by
\autoref{level_reduction_corollary}, it is isomorphic to some \(\modrep {f'}\) where \(f'\) is a newform of level \(M\).
From \autocite[Definition 4.3 \& Theorem 4.5]{edixhoven_1992}, we further have that \(\modrep{f}\) arises from a newform of weight \(2\) when it is
irreducible, and from weight \(2\) or \(4\) when it is reducible. It is essentially a modified version of Serre's Conjecture.

So, in the irreducible case, we get part (a) of \autoref{pushdown_theorem}.

For the reducible case, we can refer to \autocite[Prop 3]{carayol_1989} and see that we can have weight \(2\) if we allow level \(2M\).
The only cases where it doesn't work is the case where the form is CM type mod 2 by \(\QQ(i)\). In this case, we get forms with Fourier coefficients \(a_p\)
congruent to \(0\) modulo \((i+1)\) for every prime \(p \equiv 3 \pmod{4}\). So, in this case, \(\mc E(f) \ge 0.5\). This concludes the proof
of \autoref{pushdown_theorem}.

Now, for primes \(M = 2,3,5\), and 7, there are no nontrivial cusp forms of weight 2. Hence, \autoref{small_primes_corollary} follows
summarily from \autoref{pushdown_theorem}.

\section{Computations and a Proof of \autoref{stat_theorem}}
We started by running some numerical calculations done using \autocite{sagemath}, which compute an estimate of the evenness of a newform.
These give us some idea about the possible values of evenness at levels \(N \le 100\).
We use labels from \autocite{lmfdb} to refer to specific newforms and algebraic curves.

In order to prove part (a) of \autoref{stat_theorem}, we only need to worry about the levels \(M\) where \(M\) is as listed in the theorem.
Also, using \autoref{pushdown_theorem}, we can limit ourselves to the weights \(2\) and \(4\).

The proofs for \(N=2, 3, 5, 7\) are already covered in \autoref{small_primes_corollary}. For \(N=9\), the only form that we need to consider is
\(9.4.a.a\). Using numerical calculations, we can easily estimate the evenness of this form to be \(1\). Now, we can use
\autocite[Prop 4.4]{ono_taguchi_2005} to see that this is indeed the case since the other two possibilities listed there must provide an
estimation strictly lower than \(1\) in our computation.

For \(N = 15\), the two forms of level \(4\), namely \(15.4.a.a\) and \(15.4.a.b\) are congruent to \(15.2.a.a\), verified using
\autocite[Theorem 9]{murty_1997}. Now, \(15.2.a.a\) has the same L-function as the elliptic curve \(15.a\). We can easily show computationally
that the elliptic curve has an evenness of \(1\), so that proves the evenness for the modular form as well.

To do it, we directly compute the action of the Galois group of the splitting field of the elliptic curve on its 2-torsion points. This
gives rise to a continuous Galois representation
\[\bar{\rho} : \absgalgrp \to GL_2(\FF_2).\]
This is isomorphic to the representation arising from the corresponding newform, by \autocite{serre_1987}.

This method works whenever we can attach an elliptic curve to a newform. Among the levels mentioned in part (a) \autoref{stat_theorem}, it works for
all the levels except for \(N = 41\). For level \(41\), we can verify that all the newforms with trivial nebentypus are congruent to the
newform \(82.2.a.a\), which can be associated to an elliptic curve, and hence be dealt with. This concludes the proof of part (a) of
\autoref{stat_theorem}.

For part (b), we use very similar methods. We can check all of them using elliptic curves, except for \(N = 29\).
For \(N = 29\), we use a different method. The form \(29.2.a.a\) is the only one of concern, but it is associated to the genus 2 curve \(841.a\).
So we can take a look at its mod \(2\) image as computed in LMFDB, and use \autocite[Theorem 1.10]{daas_2025} to conclude that \(\mc E(f) = \frac{2}{3}\).

For part (c), once again, we use similar methods as parts (a) and (b). We use elliptic curves for most, and genus \(2\) curves for \(N = 23,
31\). This concludes the proof of \autoref{stat_theorem}.

\section{Computational data}
\label{computational-data}
Apart from the levels listed in \autoref{stat_theorem}, we can calculate the evenness for the following newforms. The \checkmark and
\texttimes indicate whether the value was formally verified.
For the unverified forms, the numerical estimation included \(a_p\)s for the first \(5000\) prime numbers.

\begin{multicols}{3}
	\begin{multicoltab}{lcc}
		\mchead[\hline][\hline]{
			Newform & Evenness
		} \\
		49.2.a.a  & 1    \checkmark \\
		49.4.a.a  & 1    \checkmark \\
		49.4.a.b  & 1    \checkmark \\
		49.4.a.c  & 1/3  \checkmark \\
		49.4.a.d  & 1/3  \checkmark \\
		49.4.a.e  & 1    \checkmark \\
		53.2.a.a  & 2/3  \texttimes \\
		53.2.a.b  & 2/3  \texttimes \\
		53.4.a.a  & 2/3  \texttimes \\
		\multicoltabbreak \\
		53.4.a.b  & 0.89 \texttimes \\
		53.2.a.a  & 2/3  \texttimes \\
		59.2.a.b  & 2/3  \texttimes \\
		59.4.a.a  & 0.89 \texttimes \\
		59.4.a.b  & 0.89 \texttimes \\
		59.4.a.c  & 1    \texttimes \\
		61.2.a.a  & 2/3  \checkmark \\
		61.4.a.a  & 6/7  \checkmark \\
		61.4.a.b  & 1    \checkmark \\
		\multicoltabbreak \\
		67.2.a.a  & 2/3  \checkmark \\
		67.2.a.b  & 0.27 \texttimes \\
		67.2.a.c  & 0.27 \texttimes \\
		67.4.a.a  & 0.69 \texttimes \\
		67.4.a.b  & 0.90 \texttimes \\
		71.2.a.a  & 0.58 \texttimes \\
		71.2.a.b  & 0.58 \texttimes \\
		71.4.a.a  & 2/3  \checkmark \\
		71.4.a.b  & 0.85 \texttimes \\
		\multicoltabbreak \\
		71.4.a.c  & 1    \checkmark \\
		73.2.a.a  & 1    \checkmark \\
		73.2.a.b  & 0.26 \texttimes \\
		73.2.a.c  & 0.26 \texttimes \\
		73.4.a.a  & 1    \checkmark \\
		73.4.a.b  & 1    \checkmark \\
		73.4.a.c  & 1    \checkmark \\
		75.2.a.a  & 2/3  \checkmark \\
		75.2.a.b  & 1    \checkmark \\
		75.2.a.c  & 2/3  \checkmark \\
		75.4.a.a  & 1    \checkmark \\
		75.4.a.b  & 1    \checkmark \\
		75.4.a.c  & 1    \checkmark \\
		75.4.a.d  & 0.66 \texttimes \\
		75.4.a.e  & 0.66 \texttimes \\
		75.4.a.f  & 1    \checkmark \\
		79.2.a.a  & 2/3  \checkmark \\
		79.2.a.b  & 0.87 \texttimes \\
		79.4.a.a  & 0.67 \texttimes \\
		79.4.a.b  & 1    \checkmark \\
		79.4.a.c  & 1    \checkmark \\
		81.2.a.a  & 0.67 \texttimes \\
		81.4.a.a  & 0.78 \texttimes \\
		81.4.a.b  & 0.67 \texttimes \\
		81.4.a.c  & 0.67 \texttimes \\
		81.4.a.d  & 0.67 \texttimes \\
		81.4.a.e  & 0.78 \texttimes \\
		83.2.a.a  & 2/3  \checkmark \\
		83.2.a.b  & 0.66 \texttimes \\
		83.4.a.a  & 0.87 \texttimes \\
		83.4.a.b  & 0.87 \texttimes \\
		87.2.a.a  & 0.66 \texttimes \\
		87.2.a.b  & 1    \checkmark \\
		87.4.a.a  & 0.89 \texttimes \\
		87.4.a.b  & 0.89 \texttimes \\
		87.4.a.c  & 1    \checkmark \\
		87.4.a.d  & 1    \checkmark \\
		89.2.a.a  & 2/3  \checkmark \\
		89.2.a.b  & 1    \checkmark \\
		89.2.a.c  & 0.66 \texttimes \\
		\multicoltabbreak \\
		89.4.a.a  & 2/3  \checkmark \\
		89.4.a.b  & 1    \checkmark \\
		89.4.a.c  & 1    \checkmark \\
		89.4.a.d  & 1    \checkmark \\
		91.2.a.a  & 2/3  \checkmark \\
		91.2.a.b  & 2/3  \checkmark \\
		91.2.a.c  & 0.67 \texttimes \\
		91.2.a.d  & 1    \checkmark \\
		91.4.a.a  & 0.90 \texttimes \\
		91.4.a.b  & 1    \checkmark \\
		91.4.a.c  & 0.90 \texttimes \\
		91.4.a.d  & 1    \checkmark \\
		93.2.a.a  & 0.67 \texttimes \\
		93.2.a.b  & 1    \checkmark \\
		93.4.a.a  & 2/3  \checkmark \\
		93.4.a.b  & 0.67 \texttimes \\
		93.4.a.c  & 0.90 \texttimes \\
		93.4.a.d  & 1    \checkmark \\
		93.4.a.e  & 1    \checkmark \\
		93.4.a.f  & 1    \checkmark \\
		97.2.a.a  & 0.11 \texttimes \\
		97.2.a.b  & 1    \checkmark \\
		97.4.a.a  & 1    \checkmark \\
		97.4.a.b  & 1    \checkmark \\
		99.2.a.a  & 1    \checkmark \\
		99.2.a.b  & 1    \checkmark \\
		99.2.a.c  & 1    \checkmark \\
		99.2.a.d  & 2/3  \checkmark \\
		99.4.a.a  & 1    \checkmark \\
		99.4.a.b  & 1    \checkmark \\
		99.4.a.c  & 1    \checkmark \\
		99.4.a.d  & 0.66 \texttimes \\
		99.4.a.e  & 1    \checkmark \\
		99.4.a.f  & 1    \checkmark \\
		99.4.a.g  & 1    \checkmark \\
		101.2.a.a & 2/3  \checkmark \\
		101.2.a.b & 0.84 \texttimes \\
		101.4.a.a & 0.95 \texttimes \\
		101.4.a.b & 1    \checkmark \\
		103.2.a.a & 0.25 \texttimes \\
		\multicoltabbreak \\
		103.2.a.b & 0.71 \texttimes \\
		103.4.a.a & 0.89 \texttimes \\
		103.4.a.b & 1    \checkmark \\
		107.2.a.a & 0.25 \texttimes \\
		107.2.a.b & 0.75 \texttimes \\
		107.4.a.a & 0.91 \texttimes \\
		107.4.a.b & 0.91 \texttimes \\
		109.2.a.a & 2/3  \checkmark \\
		109.2.a.b & 0.13 \texttimes \\
		109.2.a.c & 0.72 \texttimes \\
		109.4.a.a & 0.89 \texttimes \\
		109.4.a.b & 1    \checkmark \\
		113.2.a.a & 1    \checkmark \\
		113.2.a.b & 1    \checkmark \\
		113.2.a.c & 0.11 \texttimes \\
		113.2.a.d & 0.11 \texttimes \\
		113.4.a.a & 1    \checkmark \\
		113.4.a.b & 1    \checkmark \\
		115.2.a.a & 2/3  \checkmark \\
		115.2.a.b & 0.67 \texttimes \\
		115.2.a.c & 1    \checkmark \\
		115.4.a.a & 1    \checkmark \\
		115.4.a.b & 2/3  \checkmark \\
		115.4.a.c & 0.62 \texttimes \\
		115.4.a.d & 0.88 \texttimes \\
		115.4.a.e & 0.96 \texttimes \\
		115.4.a.f & 1    \checkmark \\
		119.2.a.a & 0.59 \texttimes \\
		119.2.a.b & 1    \checkmark \\
		119.4.a.a & 1    \checkmark \\
		119.4.a.b & 1    \checkmark \\
		119.4.a.c & 1    \checkmark \\
		119.4.a.d & 1    \checkmark \\
		119.4.a.e & 1    \checkmark \\
		121.2.a.a & 2/3  \checkmark \\
		121.2.a.b & 2/3  \checkmark \\
		121.2.a.c & 2/3  \checkmark \\
		121.2.a.d & 2/3  \checkmark \\
		121.4.a.a & 1    \checkmark \\
		121.4.a.b & 0.68 \texttimes \\
		\multicoltabbreak \\
		121.4.a.c & 0.66 \texttimes \\
		121.4.a.d & 0.66 \texttimes \\
		121.4.a.e & 0.68 \texttimes \\
		121.4.a.f & 0.69 \texttimes \\
		121.4.a.g & 0.69 \texttimes \\
		121.4.a.h & 0.68 \texttimes \\
		123.2.a.a & 2/3  \checkmark \\
		123.2.a.b & 2/3  \checkmark \\
		123.2.a.c & 0.68 \texttimes \\
		123.2.a.d & 1    \checkmark \\
		123.4.a.a & 1    \checkmark \\
		123.4.a.b & 1    \checkmark \\
		123.4.a.c & 1    \checkmark \\
		123.4.a.d & 1    \checkmark \\
		125.2.a.a & 0.27 \texttimes \\
		125.2.a.b & 0.27 \texttimes \\
		125.2.a.c & 0.60 \texttimes \\
		125.4.a.a & 0.60 \texttimes \\
		125.4.a.b & 0.70 \texttimes \\
		125.4.a.c & 0.70 \texttimes \\
		125.4.a.d & 0.60 \texttimes \\
		127.2.a.a & 0.12 \texttimes \\
		127.2.a.b & 0.66 \texttimes \\
		127.4.a.a & 2/3  \checkmark \\
		127.4.a.b & 1    \checkmark \\
		127.4.a.c & 1    \checkmark \\
		129.2.a.a & 2/3  \checkmark \\
		129.2.a.b & 1    \checkmark \\
		129.2.a.c & 0.67 \texttimes \\
		129.2.a.d & 0.89 \texttimes \\
		129.4.a.a & 1    \checkmark \\
		129.4.a.b & 2/3  \checkmark \\
		129.4.a.c & 0.67 \texttimes \\
		129.4.a.d & 1    \checkmark \\
		129.4.a.e & 1    \checkmark \\
		129.4.a.f & 1    \checkmark \\
		131.2.a.a & 2/3  \checkmark \\
		131.2.a.b & 0.73 \texttimes \\
		131.4.a.a & 0.91 \texttimes \\
		131.4.a.b & 0.91 \texttimes \\
		\multicoltabbreak \\
		133.2.a.a & 0.26 \texttimes \\
		133.2.a.b & 0.26 \texttimes \\
		133.2.a.c & 0.26 \texttimes \\
		133.2.a.d & 1    \checkmark \\
		133.4.a.a & 1    \checkmark \\
		133.4.a.b & 0.71 \texttimes \\
		133.4.a.c & 0.71 \texttimes \\
		133.4.a.d & 1    \checkmark \\
		133.4.a.e & 1    \checkmark \\
		135.2.a.a & 2/3  \checkmark \\
		135.2.a.b & 2/3  \checkmark \\
		135.2.a.c & 0.67 \texttimes \\
		135.2.a.d & 0.67 \texttimes \\
		135.4.a.a & 2/3  \checkmark \\
		135.4.a.b & 2/3  \checkmark \\
		135.4.a.c & 2/3  \checkmark \\
		135.4.a.d & 2/3  \checkmark \\
		135.4.a.e & 0.89 \texttimes \\
		135.4.a.f & 0.67 \texttimes \\
		135.4.a.g & 0.67 \texttimes \\
		135.4.a.h & 0.89 \texttimes \\
		137.2.a.a & 0.05 \texttimes \\
		137.2.a.b & 1    \checkmark \\
		137.4.a.a & 1    \checkmark \\
		137.4.a.b & 1    \checkmark \\
		139.2.a.a & 2/3  \checkmark \\
		139.2.a.b & 0.12 \texttimes \\
		139.2.a.c & 0.71 \texttimes \\
		139.4.a.a & 0.66 \texttimes \\
		139.4.a.b & 0.88 \texttimes \\
		139.4.a.c & 0.88 \texttimes \\
		143.2.a.a & 2/3  \checkmark \\
		143.2.a.b & 0.62 \texttimes \\
		143.2.a.c & 1    \checkmark \\
		143.4.a.a & 0.96 \texttimes \\
		143.4.a.b & 0.99 \texttimes \\
		143.4.a.c & 1    \checkmark \\
		143.4.a.d & 1    \checkmark \\
		147.2.a.a & 1    \checkmark \\
		147.2.a.b & 2/3  \checkmark \\
		\multicoltabbreak \\
		147.2.a.c & 2/3  \checkmark \\
		147.2.a.d & 1    \checkmark \\
		147.2.a.e & 1    \checkmark \\
		147.4.a.a & 2/3  \checkmark \\
		147.4.a.b & 2/3  \checkmark \\
		147.4.a.c & 1    \checkmark \\
		147.4.a.d & 1    \checkmark \\
		147.4.a.e & 1    \checkmark \\
		147.4.a.f & 1    \checkmark \\
		147.4.a.g & 1    \checkmark \\
		147.4.a.h & 1    \checkmark \\
		147.4.a.i & 1    \checkmark \\
		147.4.a.j & 1    \checkmark \\
		147.4.a.k & 1    \checkmark \\
		147.4.a.l & 0.90 \texttimes \\
		147.4.a.m & 0.90 \texttimes \\
		149.2.a.a & 0.12 \texttimes \\
		149.2.a.b & 0.63 \texttimes \\
		149.4.a.a & 0.87 \texttimes \\
		149.4.a.b & 1    \checkmark \\
		151.2.a.a & 0.13 \texttimes \\
		151.2.a.b & 0.55 \texttimes \\
		151.2.a.c & 0.62 \texttimes \\
		151.4.a.a & 0.87 \texttimes \\
		153.2.a.a & 2/3  \checkmark \\
		153.2.a.b & 2/3  \checkmark \\
		153.2.a.c & 1    \checkmark \\
		153.2.a.d & 2/3  \checkmark \\
		153.2.a.e & 1    \checkmark \\
		153.4.a.a & 1    \checkmark \\
		153.4.a.b & 1    \checkmark \\
		153.4.a.c & 1    \checkmark \\
		153.4.a.d & 1    \checkmark \\
		153.4.a.e & 0.67 \texttimes \\
		153.4.a.f & 1    \checkmark \\
		153.4.a.g & 1    \checkmark \\
		153.4.a.h & 1    \checkmark \\
		153.4.a.i & 1    \checkmark \\
		157.2.a.a & 0.03 \texttimes \\
		157.2.a.b & 0.67 \texttimes \\
		\multicoltabbreak \\
		157.4.a.a & 0.85 \texttimes \\
		157.4.a.b & 1    \checkmark \\
		159.2.a.a & 0.61 \texttimes \\
		159.2.a.b & 1    \checkmark \\
		159.4.a.a & 2/3  \checkmark \\
		159.4.a.b & 2/3  \checkmark \\
		159.4.a.c & 0.88 \texttimes \\
		159.4.a.d & 0.96 \texttimes \\
		159.4.a.e & 1    \checkmark \\
		159.4.a.f & 1    \checkmark \\
		161.2.a.a & 1    \checkmark \\
		161.2.a.b & 0.67 \texttimes \\
		161.2.a.c & 1    \checkmark \\
		161.2.a.d & 1    \checkmark \\
		161.4.a.a & 1    \checkmark \\
		161.4.a.b & 1    \checkmark \\
		161.4.a.c & 1    \checkmark \\
		161.4.a.d & 1    \checkmark \\
		\multicoltabbreak \\
		163.2.a.a & 2/3  \checkmark \\
		163.2.a.b & 0.04 \texttimes \\
		163.2.a.c & 0.69 \texttimes \\
		163.4.a.a & 0.95 \texttimes \\
		163.4.a.b & 0.95 \texttimes \\
		167.2.a.a & 0.26 \texttimes \\
		167.2.a.b & 0.67 \texttimes \\
		167.4.a.a & 0.92 \texttimes \\
		169.2.a.a & 0.66 \texttimes \\
		169.2.a.b & 0.12 \texttimes \\
		169.2.a.c & 0.12 \texttimes \\
		169.4.a.a & 2/3  \checkmark \\
		169.4.a.b & 2/3  \checkmark \\
		169.4.a.c & 2/3  \checkmark \\
		169.4.a.d & 2/3  \checkmark \\
		169.4.a.e & 2/3  \checkmark \\
		169.4.a.f & 0.78 \texttimes \\
		169.4.a.g & 1    \checkmark \\
		\multicoltabbreak \\
		169.4.a.h & 0.33 \texttimes \\
		169.4.a.i & 0.66 \texttimes \\
		169.4.a.j & 0.78 \texttimes \\
		169.4.a.k & 0.61 \texttimes \\
		169.4.a.l & 0.61 \texttimes \\
		171.2.a.a & 1    \checkmark \\
		171.2.a.b & 2/3  \checkmark \\
		171.2.a.c & 2/3  \checkmark \\
		171.2.a.d & 2/3  \checkmark \\
		171.2.a.e & 1    \checkmark \\
		171.4.a.a & 1    \checkmark \\
		171.4.a.b & 1    \checkmark \\
		171.4.a.c & 1    \checkmark \\
		171.4.a.d & 2/3  \checkmark \\
		171.4.a.e & 1    \checkmark \\
		171.4.a.f & 0.89 \texttimes \\
		171.4.a.g & 1    \checkmark \\
		171.4.a.h & 1    \checkmark \\
		171.4.a.i & 1    \checkmark \\
	\end{multicoltab}
\end{multicols}

\section{Acknowledgements}
I would like to thank my advisor, Kimball Martin, for his invaluable guidance. I am grateful to him for carefully going through earlier versions
of this paper, and helping me improve on it. I also thank Preston Wake and Anna Medvedovsky for pointing me towards the relevant literature.

The calculations used in this paper were carried out using SageMath and Magma, so I would like to thank the maintainers of these projects.
The LMFDB was instrumental in both the observation that eventually led to this paper, and also throughout the research phase, and I am
grateful to the maintainers of that fantastic project as well.

\printbibliography

@article{hatada_1998,
    title = {On Classical and \(\ell\)\hyphen Adic Modular Forms of Levels \(N \ell^m\) and \(N\)},
  author = {Hatada, Kazuyuki},
  date = {2001-03-01},
  journaltitle = {Journal of Number Theory},
  volume = {87},
  number = {1},
  pages = {1--14},
  issn = {0022-314X},
  doi = {10.1006/jnth.2000.2583},
}

@article{serre_1987,
  title = {Sur Les Représentations Modulaires de Degré 2 de \\ \(Gal(\bar{\QQ}/\QQ)\)},
  author = {Serre, Jean-Pierre},
  date = {1987-01},
  journaltitle = {Duke Mathematical Journal},
  volume = {54},
  number = {1},
  pages = {179--230},
  publisher = {Duke University Press},
  issn = {0012-7094, 1547-7398},
  doi = {10.1215/S0012-7094-87-05413-5},
}

@article{edixhoven_1992,
  title = {The Weight in Serre's Conjectures on Modular Forms},
  author = {Edixhoven, Bas},
  date = {1992-12},
  journaltitle = {Inventiones Mathematicae},
  volume = {109},
  number = {1},
  pages = {563--594},
  issn = {0020-9910, 1432-1297},
  doi = {10.1007/BF01232041},
}

@incollection{ribet_1994,
  title = {Report on Mod \(\ell\) Representations of \(\absgalgrp\)},
  booktitle = {Proceedings of Symposia in Pure Mathematics},
  author = {Ribet, Kenneth A.},
  editor = {Jannsen, Uwe and Kleiman, Steven and Serre, Jean-Pierre},
  date = {1994},
  volume = {55.2},
  pages = {639--676},
  publisher = {American Mathematical Society},
  location = {Providence, Rhode Island},
  doi = {10.1090/pspum/055.2/1265566},
  isbn = {978-0-8218-1637-0 978-0-8218-9356-2},
}

@article{ono_taguchi_2005,
  title = {2-Adic Properties of Certain Modular Forms and Their Applications to Arithmetic Functions},
  author = {Ono, Ken and Taguchi, Yuichiro},
  date = {2005-03},
  journaltitle = {International Journal of Number Theory},
  shortjournal = {Int. J. Number Theory},
  volume = {01},
  number = {01},
  pages = {75--101},
  issn = {1793-0421, 1793-7310},
  doi = {10.1142/S1793042105000066},
}

@online{daas_2025,
  title = {Congruence Conditions for the Mod \(\lambda\) Values of the {{Fourier}} Coefficients of Classical Eigenforms},
  author = {Daas, Michael A.},
  date = {2025-06-10},
  eprint = {2506.08865},
  eprinttype = {arXiv},
  eprintclass = {math},
  doi = {10.48550/arXiv.2506.08865},
}

@inproceedings{swinnerton_dyer_1973,
  title = {On \(\ell\)-{{ADIC Representations}} and {{Congruences}} for {{Coefficients}} of {{Modular Forms}}},
  booktitle = {Modular {{Functions}} of {{One Variable III}}},
  author = {Swinnerton-Dyer, H. P. F.},
  editor = {Kuijk, Willem and Serre, Jean-Pierre},
  date = {1973},
  series = {Lecture {{Notes}} in {{Mathematics}}},
  pages = {1--55},
  publisher = {Springer},
  location = {Berlin, Heidelberg},
  doi = {10.1007/978-3-540-37802-0_1},
  isbn = {978-3-540-37802-0},
}

@inproceedings{swinnerton_dyer_1977,
  title = {On \(\ell\)-Adic Representations and Congruences for Coefficients of Modular Forms ({{II}})},
  booktitle = {Modular {{Functions}} of One {{Variable V}}},
  author = {Swinnerton-Dyer, H. P. F.},
  editor = {Serre, Jean-Pierre and Zagier, Don Bernard},
  date = {1977},
  series = {Lecture {{Notes}} in {{Mathematics}}},
  pages = {63--90},
  publisher = {Springer},
  location = {Berlin, Heidelberg},
  doi = {10.1007/BFb0063945},
  isbn = {978-3-540-37291-2},
}

@incollection{ono_1994,
  title = {Congruences on the {{Fourier}} Coefficients of Modular Forms on \(\Gamma_0(N)\)},
  booktitle = {Contemporary {{Mathematics}}},
  author = {Ono, Ken},
  editor = {Andrews, George E. and Bressoud, David M. and Parson, L. Alayne},
  date = {1994},
  volume = {166},
  pages = {93--105},
  publisher = {American Mathematical Society},
  location = {Providence, Rhode Island},
  doi = {10.1090/conm/166/01643},
  isbn = {978-0-8218-5173-9 978-0-8218-7757-9},
}

@incollection{deligne_1971,
  title = {Formes Modulaires et Représentations \(\ell\)-Adiques},
  booktitle = {Sémi\hyphen naire Bourbaki Vol. 1968/69 Exposés 347-363},
  author = {Deligne, Pierre},
  date = {1971},
  volume = {179},
  pages = {139--172},
  publisher = {Springer Berlin Heidelberg},
  doi = {10.1007/BFb0058810},
  isbn = {978-3-540-05356-9},
}

@article{carayol_1989,
  title = {Sur Les Représentations Galoisiennes modulo \(\ell\) Attachées Aux Formes Modulaires},
  author = {Carayol, Henri},
  date = {1989-12},
  journaltitle = {Duke Mathematical Journal},
  volume = {59},
  number = {3},
  pages = {785--801},
  publisher = {Duke University Press},
  issn = {0012-7094, 1547-7398},
  doi = {10.1215/S0012-7094-89-05937-1},
}

@incollection{murty_1997,
  title = {Congruences between Modular Forms},
  booktitle = {Analytic Number Theory},
  author = {Murty, M. Ram},
  editor = {Motohashi, Yoichi},
  date = {1997-10-16},
  edition = {1},
  pages = {309--320},
  publisher = {Cambridge University Press},
  doi = {10.1017/CBO9780511666179.020},
  isbn = {978-0-521-62512-8 978-0-511-66617-9},
}

@misc{lmfdb,
  shorthand    = {LMFDB},
  author       = {The {LMFDB Collaboration}},
  title        = {The {L}-functions and modular forms database},
  howpublished = {\url{https://www.lmfdb.org}},
  year         = {2026},
}

@misc{sagemath,
    shorthand    = {Sage},
    author       = {{The Sage Developers}},
    title        = {{S}age{M}ath, the {S}age {M}athematics {S}oftware {S}ystem},
    url          = {https://www.sagemath.org},
    version      = {9.5},
    year         = {2022},
}
\end{document}